\documentclass[11pt]{article}
\usepackage{geometry}
\usepackage{amssymb}
\usepackage{amsmath}
\usepackage{amsfonts}
\usepackage{algorithm}
\usepackage{algpseudocode}
\usepackage{bbold}
\usepackage{xcolor}
\usepackage{soul}
\usepackage{graphicx}
\usepackage{wrapfig}
\usepackage{subcaption}
\usepackage[nopar]{lipsum}
\usepackage{amsthm}
\usepackage{euscript}
\usepackage{blkarray}
\usepackage{float,epsfig, floatflt,here}
\usepackage{pgf,tikz,amsmath}
\usepackage{tikzit}
\usetikzlibrary{calc}
\usetikzlibrary{patterns}

\tikzstyle{Boxy}=[fill=black, draw=black, shape=rectangle]

\tikzstyle{new edge style 0}=[-]
\tikzstyle{Line}=[-, draw={rgb,255: red,128; green,128; blue,128}]

\usepackage{blkarray}
\usepackage{caption}
\usepackage{enumerate}
\usetikzlibrary{arrows}
\usepackage{bbm} 
\usepackage{algorithm}
\usepackage{algpseudocode}
\usepackage[pagebackref,colorlinks=true,
linkcolor=blue,
urlcolor=red,colorlinks,
citecolor=green]{hyperref}
\usepackage{cleveref}
\usepackage{amsthm}
\usetikzlibrary{decorations.markings}
\tikzset{->-/.style={decoration={
			markings,
			mark=at position #1 with {\arrow{>}}},postaction={decorate}}}

\usepackage{makecell}
\numberwithin{equation}{section}
\newtheorem{thm}{Theorem}[section]
\newtheorem{prop}[thm]{Proposition}
\newtheorem{lem}[thm]{Lemma}
\newtheorem{cor}[thm]{Corollary}
\newtheorem{exm}[thm]{Example}
\newtheorem{df}[thm]{Definition}

\newcommand {\R} {\mathbb{R}}

\newcommand {\rv} {\mathbb{R}^{V}}

\makeatletter
\@namedef{subjclassname@1991}{Subject}
\@namedef{subjclassname@2000}{Subject}
\@namedef{subjclassname@2010}{Subject}
\@namedef{subjclassname@2020}{Subject}
\makeatother

\title{On the Spectra of Threshold Hypergraphs}
\author{ Anirban Banerjee\footnote{\scriptsize Department of Mathematics and Statistics, IISER Kolkata, Kolkata, India (anirban.banerjee@iiserkol.ac.in).} 
\and Rajiv Mishra\footnote{\scriptsize Department of Mathematics and Statistics, IISER Kolkata, Kolkata, India (rm20rs017@iiserkol.ac.in).} \and Samiron Parui\footnote{\scriptsize Department of Mathematics and Statistics, IISER Kolkata, Kolkata, India (samironparui@gmail.com)}
}

\begin{document}
\maketitle
	
	
	\maketitle
	\begin{abstract}
Starting with an isolated vertex, here we construct a threshold hypergraph by repeatedly adding an isolated vertex or a $k$-dominating vertex set. We represent a threshold hypergraph by a string of non-negative integers and find the Laplacian spectrum of threshold hypergraphs from their string representation. We also compute the complete Laplacian spectrum of certain threshold hypergraphs from the Ferrer’s diagram of their degree sequences. We show that the Laplacian spectra of threshold hypergraphs are $r$-integral, i.e., integral multiple of $r$, for some $r\in \mathbb{Q}$. We also construct another class of hypergraphs whose Laplacian spectra are $r$-integral.
	\end{abstract}
	\section{Introduction}
	 In $1977$, Chv\`atal and Hammer \cite{chvtal1977aggregation} introduced threshold graphs to investigate the set packing problems. Henderson and Zalcstein \cite{Henderson1977} studied the same graphs with a different name in $1977$ to explore the process synchronization problems. Threshold graphs attracted many researchers as they have various applications in psychology, computer science, scheduling theory, etc. \cite{MahadevBook}. The spectra of different connectivity matrices associated with threshold graphs have been studied by many researchers \cite{Sciriha2011, Trevisan2013, Bapat2013, MR3672962}. 
There are many ways to produce threshold graphs \cite{golumbic2004algorithmic, MahadevBook}. In one of the constructions, a threshold graph can be created from an isolated vertex by
repeatedly adding an isolated vertex or a dominating vertex (a vertex that is connected to all other vertices). 
Here, we have generalized this construction to define threshold hypergraphs.  Researchers have attempted to generalize other methods of creating threshold graphs in order to introduce threshold hypergraphs \cite{MR791660, golumbic2004algorithmic}. The threshold hypergraphs constructed by using their method are different from ours.

A hypergraph $G$ is an ordered pair $(V, E)$ with a nonempty set $V$, called the vertex set and a hyperedge set $E$, such that $E\subseteq \mathcal{P}(V)\setminus \emptyset$. 
A hypergraph $G$ is called $m$-uniform hypergraph if $|e| = m$ for all $e\in E$. For a vertex $v\in V$, we define $E_{v}=\{e\in E:v\in e\}$ and the degree $d(v)= |E_v|$.
For any $k\in \mathbb{N}$, a set $S\subseteq V$ is called \textit{$k$-dominating set} if $|S|=k$ and for every $v\in V\setminus S$, $S\cup \{v\} \in E$.
Thus adding a $k$-dominating set $S$ to a hypergraph $G(V,E)$ is to generate a new hypergraph $G_1(V_1,E_1)$ with the vertex set  $V_1=V\cup S$ and the edge set $E_1=E\cup E_S$ where $E_S= \{S\cup \{v\}: v\in V\}$.

\begin{df}\label{T}
A hypergraph $G$ is called a threshold hypergraph if it is obtained from a single vertex by repeatedly applying any of the 
following steps an arbitrary number of times: 
\begin{enumerate}
    \item[$(1)$] Add an isolated vertex: Add a new vertex to the set of vertices without modifying the edge set. 
    \item[$(2)$] Apply $k$-domination: Add a $k$-dominating set to the existing hypergraph for some $k\in \mathbb{N}$.
    
\end{enumerate}
\end{df}
A threshold graph on $n$ vertices can be uniquely represented by a binary string $s_1 s_2 \dots s_n$ where $s_i=1$ if at the $i$-th step we add a dominating vertex, and $s_i=0$ if at the $i$-th step we add an isolated vertex \cite{MR3672962}. Since we start the construction with an isolated vertex, $s_1$ always takes the value zero. Our threshold hypergraph can be represented by a by a string $s_1\, s_2\,\ldots\,s_l$ of nonnegative integers, where 
$$
s_j=\begin{cases}
    k &\text{~if~} k \text{-domination is applied in } j\text{-th step,}\\
    0&\text{~otherwise}.
\end{cases}
$$
Here the first term $s_1$ is also zero since we start with an isolated vertex  to construct a threshold hypergraph. Now we rewrite our string concisely. We assume $m_i$ is the number of isolated vertices added between the $(i-1)$-th and $i$-th domination, $k_i$-dominating set is added at 
the $i$-th domination,  and $F_i$ is the set of all the hyperedges added in the $i$-th domination.
Therefore, we can represent a connected threshold hypergraph by the following string $$0^{m_1}\,k_1\,0^{m_2}\,k_2\,\ldots\,0^{m_d}\,k_d $$
where $0^{m}$ denotes $m$ consecutive zeros in the string, and $d$ denotes the total number of dominations applied to construct the threshold hypergraph. Clearly, $m_1\geq 1$.
For a threshold graph, the Laplacian spectrum can be computed directly from its binary string representation \cite{MR3672962}. In this work, one of our main results is to show that the complete Laplacian spectrum of a threshold hypergraph can be obtained from its string $0^{m_1}\,k_1\,0^{m_2}\,k_2\,\ldots\,0^{m_d}\,k_d$. This is addressed in \Cref{spectra-threshold}. 
\begin{thm}\label{spectra-threshold}
Let $G=0^{m_1}\,k_1\,0^{m_2}\,k_2\,\ldots\,0^{m_d}\,k_d $ be a threshold hypergraph. Then the eigenvalues of $L_G$ are as follows
\begin{enumerate}
    \item[$(i)$]  $0$ with multiplicity $1$;
    \item[$(ii)$] $d-i+1$ with multiplicity $m_i$, for each $i=2,\ldots,d$;
       \item[$(iii)$]   $d$ with multiplicity $m_1-1$;
       \item[$(iv)$]  $\frac{1}{k_i} \left(\sum\limits_{j=1}^i m_j+\sum\limits_{j=1}^i k_j\right)+d-i$ with multiplicity $1$, for each $i=1,\ldots, d$;
         \item[$(v)$]  $\frac{k_i+1}{k_i} \left(\sum\limits_{j=1}^i m_j+\sum\limits_{j=1}^{i-1} k_j\right)+d-i$ with multiplicity $k_i-1$, for each $i=1,\ldots,d$.
\end{enumerate}
 \end{thm}
We prove \Cref{spectra-threshold} in \Cref{sec_thld}.

The Laplacian spectrum of a threshold graph can also be computed from the Ferrer’s diagram of its degree sequence \cite{MR2797201}.
\begin{df}[Ferrer's diagram]{\rm\cite[p. $148$]{MR2797201}}
Suppose that $S=\{a_1, a_2,\ldots, a_n\}$ is a finite collection of non-negative integers such that $a_i\ge a_{i+1}$ for all $i = 1, \dots, n-1$. The Ferrer's diagram of $S$ is a diagram consisting of $n$ rows of boxes such that
\begin{enumerate}
    \item [$(1)$]the rows are left-justified, that is, start from the left, and
    \item [$(2)$]the $i$-th row $($from above$)$ contains $a_i$ number of boxes. 
\end{enumerate}
\end{df}

In Ferrer's diagram, a \textit{block} is a collection of all the rows with the same number of boxes. The number of rows within a block is called the \textit{height} of the block. The \textit{width} of a block is the number of boxes in a row in that block (see \Cref{fig:ferrer}). 

\begin{exm}\rm \label{exm:ferrer}
Let $G$ be a threshold hypergraph whose hyperedge set $E$ is as follows
\begin{eqnarray*}
    E&=&\{\{1 ,2, 3, 4\}, \{1, 6, 7, 8\}, \{2, 6, 7, 8\},\{3, 6, 7, 8\}, \{4, 6, 7, 8\}, \{5, 6, 7, 8\}, \\
    &&\{1, 11, 12, 13\}, \{2, 11, 12, 13\}, \{3, 11, 12, 13\}, \{4, 11, 12, 13\}, \{5, 11, 12, 13\}, \{6, 11, 12, 13\}, \\
    &&\{7, 11, 12, 13\}, \{8, 11, 12, 13\}, \{9, 11, 12, 13\}, \{10, 11, 12, 13\}, \{1, 14, 15, 16\}, \{2, 14, 15, 16\},\\
    &&\{3, 14, 15, 16\}, \{4, 14, 15, 16\}, \{5, 14, 15, 16\}, \{6, 14, 15, 16\}, \{7, 14, 15, 16\}, \{8, 14, 15, 16\}, \\
    &&\{9, 14, 15, 16\}, \{10, 14, 15, 16\}, \{11, 14, 15, 16\}, \{12, 14, 15, 16\}, \{13, 14, 15, 16\}\};
\end{eqnarray*}
After arranging the degree sequence in decreasing order, we have
\[\{13,13,13,11,11,11,7,7,7,4,4,4,4,3,2,2\}\]
The Ferrer's diagram of the degree sequence is shown in \Cref{fig:ferrer}.
\end{exm}
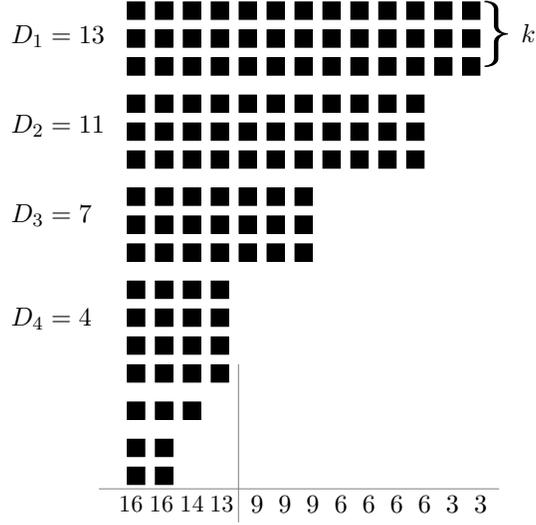
\begin{figure}[ht] 
\centering
\scalebox{.9}{\begin{tikzpicture}[scale=0.55]
	\begin{pgfonlayer}{nodelayer}
	\draw (-4.9,7.14) node[anchor=north west, scale=2.54] {$\}$};
		\draw (-3.4,6.14) node[anchor=north west, scale=1] {$k$};
	\draw (-17.1,6.1) node[anchor=north west, scale=1] {$D_1=13$};

		\draw (-17.1,3.7) node[anchor=north west, scale=1] {$D_2=11$};

		\draw (-17.1,1.3) node[anchor=north west, scale=1] {$D_3=7$};

		\draw (-17.1,-1.5) node[anchor=north west, scale=1] {$D_4=4$};

		\node [style=Boxy] (0) at (-13.5, 6.25) {};
		\node [style=Boxy] (1) at (-12.75, 6.25) {};
		\node [style=Boxy] (2) at (-12, 6.25) {};
		\node [style=Boxy] (3) at (-11.25, 6.25) {};
		\node [style=Boxy] (4) at (-10.5, 6.25) {};
		\node [style=Boxy] (5) at (-9.75, 6.25) {};
		\node [style=Boxy] (6) at (-9, 6.25) {};
		\node [style=Boxy] (7) at (-8.25, 6.25) {};
		\node [style=Boxy] (8) at (-7.5, 6.25) {};
		\node [style=Boxy] (9) at (-6.75, 6.25) {};
		\node [style=Boxy] (10) at (-6, 6.25) {};
		\node [style=Boxy] (11) at (-5.25, 6.25) {};
		\node [style=Boxy] (12) at (-4.5, 6.25) {};
		\node [style=Boxy] (13) at (-13.5, 5.5) {};
		\node [style=Boxy] (14) at (-12.75, 5.5) {};
		\node [style=Boxy] (15) at (-12, 5.5) {};
		\node [style=Boxy] (16) at (-11.25, 5.5) {};
		\node [style=Boxy] (17) at (-10.5, 5.5) {};
		\node [style=Boxy] (18) at (-9.75, 5.5) {};
		\node [style=Boxy] (19) at (-9, 5.5) {};
		\node [style=Boxy] (20) at (-8.25, 5.5) {};
		\node [style=Boxy] (21) at (-7.5, 5.5) {};
		\node [style=Boxy] (22) at (-6.75, 5.5) {};
		\node [style=Boxy] (23) at (-6, 5.5) {};
		\node [style=Boxy] (24) at (-5.25, 5.5) {};
		\node [style=Boxy] (25) at (-4.5, 5.5) {};
		\node [style=Boxy] (26) at (-13.5, 4.75) {};
		\node [style=Boxy] (27) at (-12.75, 4.75) {};
		\node [style=Boxy] (28) at (-12, 4.75) {};
		\node [style=Boxy] (29) at (-11.25, 4.75) {};
		\node [style=Boxy] (30) at (-10.5, 4.75) {};
		\node [style=Boxy] (31) at (-9.75, 4.75) {};
		\node [style=Boxy] (32) at (-9, 4.75) {};
		\node [style=Boxy] (33) at (-8.25, 4.75) {};
		\node [style=Boxy] (34) at (-7.5, 4.75) {};
		\node [style=Boxy] (35) at (-6.75, 4.75) {};
		\node [style=Boxy] (36) at (-6, 4.75) {};
		\node [style=Boxy] (37) at (-5.25, 4.75) {};
		\node [style=Boxy] (38) at (-4.5, 4.75) {};
		\node [style=Boxy] (39) at (-13.5, 3.75) {};
		\node [style=Boxy] (40) at (-12.75, 3.75) {};
		\node [style=Boxy] (41) at (-12, 3.75) {};
		\node [style=Boxy] (42) at (-11.25, 3.75) {};
		\node [style=Boxy] (43) at (-10.5, 3.75) {};
		\node [style=Boxy] (44) at (-9.75, 3.75) {};
		\node [style=Boxy] (45) at (-9, 3.75) {};
		\node [style=Boxy] (46) at (-8.25, 3.75) {};
		\node [style=Boxy] (47) at (-7.5, 3.75) {};
		\node [style=Boxy] (48) at (-6.75, 3.75) {};
		\node [style=Boxy] (49) at (-6, 3.75) {};
		\node [style=Boxy] (50) at (-13.5, 3) {};
		\node [style=Boxy] (51) at (-12.75, 3) {};
		\node [style=Boxy] (52) at (-12, 3) {};
		\node [style=Boxy] (53) at (-11.25, 3) {};
		\node [style=Boxy] (54) at (-10.5, 3) {};
		\node [style=Boxy] (55) at (-9.75, 3) {};
		\node [style=Boxy] (56) at (-9, 3) {};
		\node [style=Boxy] (57) at (-8.25, 3) {};
		\node [style=Boxy] (58) at (-7.5, 3) {};
		\node [style=Boxy] (59) at (-6.75, 3) {};
		\node [style=Boxy] (60) at (-6, 3) {};
		\node [style=Boxy] (61) at (-13.5, 2.25) {};
		\node [style=Boxy] (62) at (-12.75, 2.25) {};
		\node [style=Boxy] (63) at (-12, 2.25) {};
		\node [style=Boxy] (64) at (-11.25, 2.25) {};
		\node [style=Boxy] (65) at (-10.5, 2.25) {};
		\node [style=Boxy] (66) at (-9.75, 2.25) {};
		\node [style=Boxy] (67) at (-9, 2.25) {};
		\node [style=Boxy] (68) at (-8.25, 2.25) {};
		\node [style=Boxy] (69) at (-7.5, 2.25) {};
		\node [style=Boxy] (70) at (-6.75, 2.25) {};
		\node [style=Boxy] (71) at (-6, 2.25) {};
		\node [style=Boxy] (72) at (-13.5, 1.25) {};
		\node [style=Boxy] (73) at (-12.75, 1.25) {};
		\node [style=Boxy] (74) at (-12, 1.25) {};
		\node [style=Boxy] (75) at (-11.25, 1.25) {};
		\node [style=Boxy] (76) at (-10.5, 1.25) {};
		\node [style=Boxy] (77) at (-9.75, 1.25) {};
		\node [style=Boxy] (78) at (-9, 1.25) {};
		\node [style=Boxy] (79) at (-13.5, 0.5) {};
		\node [style=Boxy] (80) at (-12.75, 0.5) {};
		\node [style=Boxy] (81) at (-12, 0.5) {};
		\node [style=Boxy] (82) at (-11.25, 0.5) {};
		\node [style=Boxy] (83) at (-10.5, 0.5) {};
		\node [style=Boxy] (84) at (-9.75, 0.5) {};
		\node [style=Boxy] (85) at (-9, 0.5) {};
		\node [style=Boxy] (86) at (-13.5, -0.25) {};
		\node [style=Boxy] (87) at (-12.75, -0.25) {};
		\node [style=Boxy] (88) at (-12, -0.25) {};
		\node [style=Boxy] (89) at (-11.25, -0.25) {};
		\node [style=Boxy] (90) at (-10.5, -0.25) {};
		\node [style=Boxy] (91) at (-9.75, -0.25) {};
		\node [style=Boxy] (92) at (-9, -0.25) {};
		\node [style=Boxy] (93) at (-13.5, -1.25) {};
		\node [style=Boxy] (94) at (-12.75, -1.25) {};
		\node [style=Boxy] (95) at (-12, -1.25) {};
		\node [style=Boxy] (96) at (-11.25, -1.25) {};
		\node [style=Boxy] (97) at (-13.5, -2) {};
		\node [style=Boxy] (98) at (-12.75, -2) {};
		\node [style=Boxy] (99) at (-12, -2) {};
		\node [style=Boxy] (100) at (-11.25, -2) {};
		\node [style=Boxy] (101) at (-13.5, -2.75) {};
		\node [style=Boxy] (102) at (-12.75, -2.75) {};
		\node [style=Boxy] (103) at (-12, -2.75) {};
		\node [style=Boxy] (104) at (-11.25, -2.75) {};
		\node [style=Boxy] (105) at (-13.5, -3.5) {};
		\node [style=Boxy] (106) at (-12.75, -3.5) {};
		\node [style=Boxy] (107) at (-12, -3.5) {};
		\node [style=Boxy] (108) at (-11.25, -3.5) {};
		\node [style=Boxy] (109) at (-13.5, -4.5) {};
		\node [style=Boxy] (110) at (-12.75, -4.5) {};
		\node [style=Boxy] (111) at (-12, -4.5) {};
		\node [style=Boxy] (112) at (-13.5, -5.5) {};
		\node [style=Boxy] (113) at (-12.75, -5.5) {};
		\node [style=Boxy] (116) at (-13.5, -6.25) {};
		\node [style=Boxy] (117) at (-12.75, -6.25) {};
		\node [style=none] (118) at (-14.5, -6.6) {};
		\node [style=none] (119) at (-3.75, -6.6) {};
		\node [style=none] (120) at (-10.75, -3.25) {};
		\node [style=none] (121) at (-10.75, -7.5) {};
		\node [style=none] (122) at (-13.5, -7) {};
		\node [style=none] (123) at (-13.5, -7) {};
		\node [style=none] (124) at (-13.64, -7) {\small 16};
		\node [style=none] (125) at (-12.82, -7) {\small 16};
		\node [style=none] (126) at (-12, -7) {\small 14};
		\node [style=none] (127) at (-11.2, -7) {\small 13};
		\node [style=none] (128) at (-10.25, -7) {9};
		\node [style=none] (129) at (-9.5, -7) {9};
		\node [style=none] (130) at (-8.75, -7) {9};
		\node [style=none] (131) at (-8, -7) {6};
		\node [style=none] (132) at (-7.25, -7) {6};
		\node [style=none] (133) at (-6.5, -7) {6};
		\node [style=none] (134) at (-5.75, -7) {6};
		\node [style=none] (135) at (-4.25, -7) {3};
		\node [style=none] (136) at (-5, -7) {3};
	\end{pgfonlayer}
	\begin{pgfonlayer}{edgelayer}
		\draw [style=Line] (118.center) to (119.center);
		\draw [style=Line] (120.center) to (121.center);
	\end{pgfonlayer}
\end{tikzpicture}}
\caption{Ferrer's diagram of the degree sequence $\{13,13,13,11,11,11,7,7,7,4,4,4,4,3,2,2\}$ of a uniform $4$-threshold hypergraph, where $D_{i}$ is the width of $i$-th block from the top.}
\label{fig:ferrer}
\end{figure}
We also find the complete Laplacian spectrum of certain threshold hypergraphs from the Ferrer’s
diagram of their degree sequences. A threshold graph is a threshold hypergraph where each domination is $1$-domination. 
A threshold hypergraph is said to be a \textit{$k$-threshold hypergraph} if each domination applied is a $k$-domination. Another focus of our work is to find the complete Laplacian spectrum of a $k$-threshold hypergraph from the Ferrer’s diagram of its degree sequence. This is described in the following theorem which is proved in \Cref{sec_thld}.

\begin{thm}\label{fd}
Let $d_1\geq d_2\geq \ldots \geq d_n$ be the degree sequence of a $k$-threshold hypergraph $G$. Then the complete list of eigenvalues of $L_G$ is as follows.
\begin{enumerate}
  \item[$(i)$]  $0$ with multiplicity $1$;
       \item[$(ii)$]  $ \frac{1}{k}C_{D_{i}}$ with multiplicity $ D_i-D_{i+1}-k$ for all $i=1,\ldots, d-1$;
        \item[$(iii)$]  $ \frac{1}{k}C_{D_d}$ with multiplicity $ D_d-D_{d+1}$ if $D_d\neq d$;
        \item[$(iv)$] $ \frac{1}{k}C_i$ with multiplicity multiplicity $1$, for all $i=1,\ldots,d$;
         \item[$(v)$]   $ \frac{1}{k}[(k+1)D_{i}-i+1]$ with multiplicity $ k-1$ for all $i=1,\ldots d$.
\end{enumerate}
Here $D_{i}$ and $C_i$ are the width of $i$-th block from above and the $i$-th column sum from the left, respectively, in the Ferrer's diagram of the degree sequence, and $d=\min\{s\in\mathbb{N}: D_s <s\}-1$.
\end{thm}

A quest in spectral graph theory is: what are the graphs for which all the eigenvalues of the associated Laplacian are integers, and what can we say about the structure of those graphs? Till now, this is partially answered (see \cite{MerrisR, kirkland, Yizheng} for some progress in this direction). This 
question is also open for hypergraph. Here, we show that the Laplacian spectra of threshold hypergraphs are integral  multiple of $r$, for some $r\in \mathbb{Q}$. We also derive another class of hypergraphs whose Laplacian spectra have the same property. \Cref{sec:Integral_HyperGraphs} is dedicated to this.

\section{Spectrum of Threshold Hypergraphs}\label{sec_thld}

Let $G(V, E)$ be a finite hypergraph, that is,  $V$ is finite and let  $|e|\geq 2$ for all $e\in E$. Let $\rv$ denote the set of all real-valued functions on $V$. For $G$, the Laplacian $L_G:\rv \to \rv$ is defined as
$$(L_G f)(v)=\sum\limits_{e\in E_v}\frac{1}{|e|-1}\sum\limits_{u\in e}(f(v)-f(u))$$
for all $v\in V$. The eigenvalues of $L_G$ are real and we order the eigenvalues of $L_G$ as $0=\lambda_1\leq \lambda_2\leq \ldots \leq \lambda_{|V|}$, thus $\lambda_2>0$ if the hypergraph is connected \cite{MR4208993}.
Now to prove \Cref{spectra-threshold}, we prove the following lemma.

 \begin{lem}\label{lem:k_dom_spectra}
Let $G(V,E)$ be a hypergraph and $\lambda_1=0, \lambda_2, \ldots, \lambda_{|V|}$ be the eigenvalues of $L_G$. Suppose $G_1(V_1,E_1)$ is a hypergraph obtained by applying $k$-domination on $G$ for some $k\in \mathbb{N}$. Then the eigenvalues of $L_{G_1}$ are given by

\begin{enumerate}
    \item [$(i)$] $0$ with multiplicity $1$;
    \item [$(ii)$] $\lambda_i +1$ with multiplicity $1$ for $i=2,\ldots,|V|$;
    \item[$(iii)$] $\frac{1}{k}(|V|+k)$ with multiplicity $1$;
     \item[$(iv)$] $\frac{k+1}{k}|V|$ with multiplicity $k-1$.
\end{enumerate}
 \end{lem}
 \begin{proof}
  It is clear that $0$ is an eigenvalue of $L_{G_1}$ with eigenvector $\chi_{|V_1|}$. Suppose $x_i\in\R^{|V|}$ is an eigenvector corresponding to $\lambda_i$ of $L_G$ for $i=2,\ldots,|V|$. Now for $i=2,3,\ldots, |V|$, define $\bar{x}_i\in \R^{|V_1|}$ as, 
$$\bar{x}_i(v)=\begin{cases}
    x_i(v) & \text{ if } v\in V\\
    0 & \text{ otherwise}
\end{cases}$$ 
 Now note that $V_1=V\cup S$, where $S$ is the set of vertices added in the domination, $|S|=k$ and $E_1=E\cup E_S$, where $E_S=\{S\cup\{u\}:u\in V\}$. 
 
 Now for $v\in V$, $E_v=E_v\cap E\cup \{S\cup\{v\}\}$, thus
\begin{eqnarray}(L_{G_1}\bar{x_i})(v)&=&\sum\limits_{e\in E_v\cap E}\frac{1}{|e|-1}\sum\limits_{u\in e}\big(\bar{x}_i(v)-\bar{x}_i(u)\big)+\frac{1}{k}\sum\limits_{u\in S}\big(\bar{x}_i(v)-\bar{x}_i(u)\big)\nonumber\\
&=& \lambda_i \bar{x}_i(v)+\frac{1}{k}k \bar{x}_i(v)\nonumber\\
&=& \big( \lambda_i+1\big)\bar{x}_i(v).\nonumber\end{eqnarray}
Now for $v\in S$, $E_v=E_S$, thus
\begin{eqnarray}(L_{G_1}\bar{x_i})(v)&=&\sum\limits_{e\in E_S}\frac{1}{|e|-1}\sum\limits_{u\in e}\big(\bar{x}_i(v)-\bar{x}_i(u)\big)\nonumber\\
&=& -\frac{1}{k}\sum\limits_{u\in V}\bar{x}_i(u)\nonumber\\
&=& 0\nonumber\end{eqnarray}
Therefore $\lambda_i+1$ is an eigenvalue of $L_{G_1}$ for $i=2,\ldots, |V|$.

Now suppose $S=\{v_1,\ldots,v_k\}$, and for $i=2,\ldots,k$, consider the vectors $y_i\in\R^{|V_1|}$ defined by
$$y_i(v)=\begin{cases}
    1 & \text{ if }v=v_1\\
    -1 & \text{ if }v=v_i\\
    0 & \text{ otherwise.}
\end{cases}$$
Now for $v\in V$,
$$\big(L_{G_1}y_i\big)(v)=\sum\limits_{e\in E_v\cap E}\frac{1}{|e|-1}\sum\limits_{u\in e}\big(y_i(v)-y_i(u)\big)+\frac{1}{k}\sum\limits_{u\in S}\big(y_i(v)-y_i(u)\big)=0$$
and for $v\in S$
\begin{eqnarray}(L_{G_1}y_i)(v)&=&\sum\limits_{e\in E_S}\frac{1}{|e|-1}\sum\limits_{u\in e}\big(y_i(v)-y_i(u)\big)\nonumber\\
&=& \frac{1}{k}\sum\limits_{e\in E_S}\sum\limits_{u\in e}y_i(v)-y_i(u)\nonumber\\
&=& \frac{k+1}{k}|V|y_i(v)\nonumber\end{eqnarray}
Therefore $\frac{k+1}{k}|V|$ is an eigenvalue of $L_{G_1}$ with multiplicity $k-1$.

\noindent Again define $y\in\R^{|V_1|}$ by
$$y(v)=\begin{cases}
    -\frac{k}{|V|} & \text{ if }v\in V\\
    1 & \text{ if }v\in S
\end{cases}$$
Now for $v\in V$,
\begin{eqnarray}
    \big(L_{G_1}y\big)(v)&=&\sum\limits_{e\in E_v\cap E}\frac{1}{|e|-1}\sum\limits_{u\in e}\big(y(v)-y(u)\big)+\frac{1}{k}\sum\limits_{u\in S}\big(y(v)-y(u)\big)\nonumber\\
    &=& 0+\frac{1}{k}k\Big(-\frac{k}{|V|}-1\Big)\nonumber\\
    &=&\frac{1}{k}(|V|+k)y(v)\nonumber
\end{eqnarray}

and for $v\in S$
\begin{eqnarray}(L_{G_1}y)(v)&=&\sum\limits_{e\in E_S}\frac{1}{|e|-1}\sum\limits_{u\in e}\big(y(v)-y(u)\big)\nonumber\\
&=& \frac{1}{k}\sum\limits_{e\in E_S}\Big(1+\frac{k}{|V|}\Big)\nonumber\\
&=& \frac{1}{k}(|V|+k)y(v)\nonumber\end{eqnarray}
Therefore $\frac{1}{k}(|V|+k)$ is an eigenvalue of $L_{G_1}$ with multiplicity $1$.
 \end{proof}
Now we are ready to prove \Cref{spectra-threshold}.

\noindent\textbf{Proof of the \Cref{spectra-threshold}}
\begin{proof}
We use induction on the number of dominations, $d$ to prove the result. 
The case $d=1$ follows from \Cref{lem:k_dom_spectra}. We assume that the statement is true for $d=n$, as the induction hypothesis. Suppose that, for $d=n$, the threshold hypergraph is $G_n$. Let $H_n$ be the hypergraph obtained after adding all the isolated vertices just after $d=n$ and before $d=n+1$. The Laplacian spectrum of $H_{n}$ is the same as the Laplacian spectrum of $G_n$ but with an additional $m_{n+1}$ number of zeros. Now apply $k_{n+1}$-domination on $H_n$ and from \Cref{lem:k_dom_spectra}, we find the statement is also true for $d=n+1$.
\end{proof}


In a $k$-threshold hypergraph, $k_1=k_2=\ldots k_d=k$. Thus a connected $k$-threshold hypergraph is a $(k+1)$-uniform hypergraph.   By taking $k_i=k$ for each $i=1,2,\ldots,d$ in \Cref{spectra-threshold}, we can also obtain the complete spectra of $k$-threshold hypergraph.

\begin{cor}\label{cor:k_threshold}
Let $G=0^{m_1}\,k\,0^{m_2}\,k\,\ldots\,0^{m_d}\,k $ be a $k$-threshold hypergraph. Then the Laplacian spectrum of $G$ is given by
\begin{enumerate}
    \item[$(i)$] 0 with multiplicity 1;
    \item[$(ii)$] $d-i+1$ with multiplicity $m_i$, for each $i=2,3,\ldots,d$;
    \item[$(iii)$] $d$ with multiplicity $m_1-1$;
     \item[$(iv)$] $\frac{1}{k}\sum\limits_{j=1}^{i}m_j +d$ with multiplicity 1, for each $i=1,2,\ldots,d$;
      \item[$(v)$] $\frac{k+1}{k}\sum\limits_{j=1}^{i}m_j +k(i-1)+d-1$ with multiplicity $k-1$, for each $i=1,2,\ldots,d$.
\end{enumerate}
\end{cor}

As we have mentioned before the above spectrum can be directly computed only from the degree sequence of the $k$-threshold hypergraph by using the Ferrer's diagram. Now we show the same in the proof of the \Cref{fd}.\\

\noindent\textbf{Proof of the \Cref{fd}}
\begin{proof}
    Suppose that $0^{m_1}\,k\,0^{m_2}\,k\,\ldots\,0^{m_d}\,k $ be a $k$-threshold hypergraph and $d_1\geq d_2 \geq \ldots\geq d_n $ be its degree sequence. Let $D_i$  denote the width of $i$-th block  from the top in Ferrer's diagram of the degree sequence. The total number of dominations, $d$, is given by
     $$d=\min\{s\in\mathbb{N}: D_s <s\}-1$$
     since  $D_d\geq d$ and $D_{d+i}< d+1, i\in \mathbb{N}$. If $u$ be any vertex added in $i$-th domination then the degree of $u$ is given by, 
          $$d(u)=\sum\limits_{j=1}^i m_j+k(i-1)+(d-i).$$
     Therefore 
     $$D_{i}=\sum\limits_{j=1}^{d-i+1} m_j+k(d-i)+i-1 \text{ for all } i=1,2,\ldots,d.$$
     Moreover, we have
    $$m_{d-i+1}=D_{i}-D_{i+1}-k+1 \text{ for all }i=1,2,\ldots d-1,$$
and
    $$m_1=\begin{cases}
        1 &\text{ if }D_d=d,\\
        D_{d}-D_{d+1} +1& \text{ otherwise.}
    \end{cases}$$
    Now as there is $D_1$ number of columns in the Ferrer's diagram. Let $C_i$ denote the sum of all the boxes in the $i$-th column, for $i=1,2,\ldots, D_1$. Thus 
    $$C_i=\sum\limits_{j=1}^{d-i+1}m_j+k d \text{ for all } i=1,2,\ldots,d;$$
     and $$C_{D_{d-i+1}}=k(d-i+1)\text{ for all }i=2,3,\ldots,d $$
    moreover 
    $$C_{D_d}=k d\text{ if }m_1>1.$$
Therefore the proof follows from the \Cref{cor:k_threshold}.
\end{proof}
\begin{exm}\rm
Let $G$ be a $3$-threshold hypergraph described in \Cref{exm:ferrer}.
\noindent Using \Cref{fd}, the complete spectrum  of $L_G$ is given by $\{\frac{1}{3} x:x\in X\}$ where 
 $$X=\{0,6,6,9,13,13,13,14,16,16,26,26,43,43,52,52\}.$$
\end{exm}
 Note that $1$-threshold hypergraph becomes a threshold graph, and thus the result obtained for a $k$-threshold hypergraph holds for a threshold graph.

Now, we illustrate the computation of the result of \Cref{fd} in the \Cref{algo}.
\begin{algorithm}[H]
\caption{Algorithm for computing the Laplacian spectrum of a $k$-threshold hypergraph $G$ when $d\ge2$. }
\label{algo}
\begin{algorithmic}
\Require $d_1\geq d_2 \geq \ldots\geq d_n $ \Comment{ Degree sequence in non-increasing order}
\State \Comment{Draw the Ferrer's Diagram}
\State $k \gets \text{height of the first block }$.
\State \textbf{Set} $s=$ Total number of blocks. 

\While{$D_s < s$}
\State $s \gets s-1$
\EndWhile 

\State $d \gets s$\Comment{ $d$ is the number of dominations}
\State
\State $D_j \gets \text{width of }j^{th}\text{ block from the top }$ \Comment{$j=1,2,\ldots,d.$}

\State $C_j \gets j^{th} \text{ column sum}$ \Comment{$j=1,2,\ldots , d_1.$}
\State
\For{$j=1:d$}
\State $\Gamma_j \gets \frac{1}{k}((k+1)D_j-j+1);$ 
\State $\Gamma_j$ is an eigenvalue with multiplicity $k-1$. 
\EndFor

\For{$j=1:d$}
\State $\frac{1}{k}\, C_j$ is an eigenvalue with multiplicity 1.
\EndFor
\For{$j=1:d-1$}
\State $\frac{1}{k}\,C_{D_{j}}$ is an eigenvalue with multiplicity $D_j-D_{j+1}-k$.
\EndFor
\If{$D_d \neq d$} 
\State $\frac{1}{k}C_{D_d}$ is an eigenvalue with multiplicity $D_d-d$.
\EndIf
\end{algorithmic}
\end{algorithm}

\section{Integral Hypergraphs}\label{sec:Integral_HyperGraphs}

The Laplacian $L_G$ of a (hyper)graph $G$ is said to be \textit{$r$-integral} if there exists $r\in \mathbb R$ such that all the eigenvalues of  $L_G$  are the integral-multiple of $r$, that is, all the eigenvalues are of the form $rp$ where $p\in \mathbb{Z}$. The graphs for which the Laplacian is $1$-integral (Laplacian integral graphs) have been well studied in the past \cite{MerrisR, kirkland, Yizheng}.
Here, we derive some classes of hypergraphs for which the associated Laplacians are $r$-integral. 
If for a hypergraph, the associated Laplacian is $r$-integral then we say that the hypergraph is a Laplacian $r$-integral hypergraph.

 Let $G(V, E)$ be an $m$-uniform hypergraph on $n$ vertices and let $\kappa_m^n$ denote the set of all $m$ element subsets of $V$. The complement of $G$, denoted by $\overline{G}$, is the hypergraph with the same vertex set $V$ and the edge set $\kappa_m^n\setminus E$. We can check that (see \cite{MR4208993}) if $G$ is an $m$-uniform hypergraph on $n$ vertices and $0=\lambda_1\leq \lambda_2\leq \ldots \leq \lambda_n$ are the eigenvalues of $L_G$ then the eigenvalues of $L_{\overline{G}}$ are $0=\lambda_1\leq \phi_m(n)-\lambda_n\leq \ldots \leq \phi_m(n)-\lambda_2$, where  $\phi_m(n)=\frac{n}{n-1}\binom{n-1}{m-1}.$

 \begin{prop}\label{threshold_r_integral}
     Every threshold hypergraph is Laplacian $r$-integral for some $r\in \mathbb{Q}$.
 \end{prop}
 \begin{proof}
     Proof directly follows from the \Cref{spectra-threshold}.
 \end{proof}
Now we define another class of hypergraphs where each hypergraph in that class is Laplacian $r$-integral for some $r\in \mathbb{Q}$.
\begin{df}[Co-$k$-threshold hypergraph]\label{cokthreshold}
A hypergraph is called a \textit{co-$k$-threshold hypergraph} if it is constructed using the following rules recursively:

\begin{itemize}
\item [$(i)$] Any $k$-threshold hypergraph is a co-$k$-threshold hypergraph.

\item[$(ii)$]The complement of a co-$k$-threshold hypergraph is a co-$k$-threshold hypergraph.

\item[$(iii)$] The union of two vertex-disjoint co-$k$-threshold hypergraphs is a co-$k$-threshold hypergraph.
\end{itemize}
\end{df}
The converse of the point $(i)$ in the \Cref{cokthreshold} is not true. As an example, for any two $k$-threshold hypergraphs $G_1$ and $G_2$, the hypergraph $\overline{G_1\cup G_2}$ is a co-$k$-threshold hypergraph but not a $k$-threshold hypergraph.

\begin{thm}
Any Co-$k$-threshold hypergraph is Laplacian $r$-integral for some $r\in \mathbb{Q}$.
\end{thm}
\begin{proof}
    From \Cref{threshold_r_integral}, we can say that any $k$-threshold hypergraph is $r$-integral for some rational $r$. Note that we can assume that $r=\frac{1}{q}$ for some $q\in\mathbb{N}$. Also, if $G$ be any  $r$-integral hypergraph on $n$ vertices, where $r=\frac{1}{q}$ for some $q\in \mathbb{N}$, then $\overline{G}$ is $\frac{r}{n-1}$ -integral. Also, if $G_1$ be  Laplacian $\frac{1}{q_1}$-integral  and $G_2$ be Laplacian $\frac{1}{q_2}$-integral hypergraphs, then their union is  Laplacian $\frac{1}{q_1 q_2}$-integral. Thus the result follows.
\end{proof}

\section*{Acknowledgement}The work of Samiron Parui is supported by University Grant Commission, India [Beneficiary Code: 	BININ00965055 A]. The work of Rajiv Mishra is supported by the Council of Scientific$\And$Industrial Research, India[File number: 09/921(0347)/2021-EMR-I]. 
\bibliographystyle{siam}

\end{document}